\documentclass[11pt]{amsart}

\usepackage{hyperref}
\hypersetup{bookmarksdepth=3}
  
\usepackage{geometry}
\usepackage{amsmath,amssymb}
\usepackage{
mathrsfs}
\usepackage{enumerate}
\usepackage{esint}
\usepackage{mathtools}

\usepackage{tensor}

\usepackage{IEEEtrantools}

\usepackage{amssymb,amsmath,amsthm}

\newtheorem{theorem}{Theorem}

\theoremstyle{remark}\newtheorem{remark}[theorem]{Remark}

\usepackage{graphicx}

\newcommand{\abs}[1]{\left| #1 \right|}

\newcommand{\C}{{\mathbb C}}
\newcommand{\D}{\mathrm{d}}

\newcommand{\FT}{\mathcal{F}}

\newcommand{\iu}{\mathrm{i}}
\newcommand{\iv}[1]{\frac{1}{ #1}}
\newcommand{\jb}[1]{\langle #1 \rangle}

\newcommand{\norm}[1]{\left\Vert #1 \right\Vert}
\newcommand{\opT}{\mathcal{T}}

\newcommand{\R}{{\mathbb R}}

\newcommand{\set}[1]{\left\{ #1 \right\}}

\DeclareMathOperator{\supp}{supp}

\newcommand{\Z}{{\mathbb Z}}

\begin{document}

{\let\thefootnote\relax\footnote{Date: April 1st, 2019.

\textcopyright 2019 by the authors. Faithful reproduction of this article, in its entirety, by any means is permitted for noncommercial purposes.}}

\title{The global Cauchy problem for the NLS with higher order anisotropic dispersion.}


\author{L. Chaichenets}
\address{leonid chaichenets, department of mathematics, institute for analysis, karlsruhe institute of technology, 76128 karlsruhe, germany }
\email{leonid.chaichenets@kit.edu}

\author{N. Pattakos}
\address{nikolaos pattakos, department of mathematics, institute for analysis, karlsruhe institute of technology, 76128 karlsruhe, germany }
\email{nikolaos.pattakos@kit.edu}

\begin{abstract}
{We use a method developed by Strauss to obtain global wellposedness
results in the mild sense for the small data Cauchy problem in modulation spaces
$M_{p,q}^s(\R^d)$, where $q=1$ and $s\geq0$ or $q\in(1,\infty]$ and
$s>\frac{d}{q'}$ for a nonlinear Schr\"odinger equation with higher
order anisotropic dispersion and algebraic nonlinearities.}
\end{abstract}

\maketitle
\pagestyle {myheadings}

\begin{section}{introduction and main results}
\markboth{\normalsize L. Chaichenets and N. Pattakos }{\normalsize  Nonlinear Schr\"odinger equation with higher order anisotropic dispersion}

We are interested in the following Cauchy problem 
\begin{equation} 
\label{maineq}
\left\{
\begin{IEEEeqnarraybox}[][c]{rCl}
\iu \partial_t u(t,x) + \alpha \Delta u(t,x) +
\iu \beta \frac{\partial^3}{\partial x_1^3} u(t,x) +
\gamma \frac{\partial^4}{\partial x_1^4} u(t,x) +
f(u(t,x)) & = & 0, \\
u(0, \cdot) & = & u_0(\cdot),
\end{IEEEeqnarraybox}
\right.
\end{equation}
where $(t,x)=(t,x_{1},x')\in\R\times\R\times\R^{d-1}$, $d\geq2$, $\alpha\in\R\setminus\{0\}$, and $(\beta,\gamma)\in\R^{2}\setminus\{(0,0)\}$. Such PDE arise in the context of high-speed soliton transmission in long-haul optical communication system, see \cite{DP}. The case where the coefficiets $\alpha, \beta, \gamma$ are time dependent has been studied in \cite{CPS} in one dimension for the cubic nonlinearity, $f(u)=|u|^{2}u$, with initial data in $L^{2}(\R)$-based Sobolev spaces. In \cite{OB} it is proved that \eqref{maineq} with nonlinearity $f(u)=|u|^{p}u$ where 
$$p< \begin{cases}
\frac4{d-\frac12} &,\ \gamma\neq0,\\
\frac4{d-\frac13} &,\ \gamma=0,\\
\end{cases}$$
is globally wellposed in $L^{2}(\R^d)$ via Strichartz estimates and the charge conservation equation
$$\|u(t,\cdot)\|_{L^{2}(\R^d)}=\|u_{0}\|_{L^{2}(\R^d)}, \qquad \forall t\in\R.$$
In the same paper the case of initial data $u_{0}\in H^{1}_{a}(\R^d)$ is studied where 
$$H^{1}_{a}(\R^d):=\set{u\in L^{2}(\R^d) \Big| \nabla u, \partial_{x_{1}}^{2}u\in L^{2}(\R^d)}$$ is
equipped with the norm
$$\|u\|_{H^{1}_{a}(\R^d)}:=\Big(\|u\|_{L^{2}(\R^d)}^{2}+\|\nabla u\|_{L^{2}(\R^d)}^{2}+\|\partial_{x_{1}}^{2}u\|_{L^{2}(\R^d)}^{2}\Big)^{\frac12}.$$

In this paper we consider the Cauchy problem \eqref{maineq} with initial data $u_{0}$ in modulation spaces $M_{p,q}^{s}(\R^{d})$. Modulation spaces were introduced by Feichtinger in \cite{FEI} and since then, they have become canonical for both time-frequency and phase-space analysis. They provide an excellent substitute for estimates that are known to fail on Lebesgue spaces. To state the definition of a modulation space we need to fix some notation. We will denote by $S'(\R^{d})$ the space of tempered distributions. Let $Q_{0}$ be the unit cube with center the origin in $\R^{d}$ and its translations $Q_k \coloneqq Q_0 + k$ for all $k\in\mathbb Z^{d}.$ Consider a partition of unity $\{\sigma_{k}=\sigma_{0}(\cdot-k)\}_{k\in\mathbb Z^{d}}\subset C^{\infty}(\mathbb R^{d})$ satisfying 
\begin{itemize}
\item
$ \exists c > 0: \,
\forall \eta \in Q_{0}: \,
|\sigma_{0}(\eta)| \geq c$,
\item
$
\mbox{supp}(\sigma_{0}) \subseteq \{\xi\in\R^{d}:|\xi|<\sqrt{d}\}=:B(0,\sqrt{d})$,
\end{itemize}
and define the isometric decomposition operators
\begin{equation}
\label{iso}
\Box_k \coloneqq \FT^{(-1)} \sigma_k \FT, \qquad
\forall k \in \Z^d,
\end{equation}
where $\mathcal F$ denotes the Fourier transform in $\R^d$. Then the norm of a tempered distribution $f\in S'(\R^{d})$ in the modulation space $M^{s}_{p,q}(\mathbb R^{d})$, where $s\in\mathbb R, 1\leq p,q\leq\infty$, is given by
\begin{equation}
\label{def}
\|f\|_{M^{s}_{p,q}}:=\Big\|\Big\{\langle k\rangle^{s}\|\Box_{k}f\|_{p}\Big\}_{k\in\Z^{d}}\Big\|_{l^{q}(\Z^{d})},
\end{equation}
where we denote by $\langle k\rangle=1+|k|$ the Japanese bracket. It can be proved that different choices of the function $\sigma_{0}$ lead to equivalent norms in $M^{s}_{p,q}(\mathbb R^{d})$ (see e.g. \cite[Proposition 2.9]{CHA}). When $s=0$ we denote the space $M^{0}_{p,q}(\mathbb R^{d})$ by $M_{p,q}(\mathbb R^{d})$. In the special case where $p=q=2$ we have $M_{2,2}^{s}(\R^{d})=H^{s}(\R^{d})$ the usual Sobolev spaces.

For $\alpha \in \R$ we define the weighted mixed-norm space
\begin{equation*}
L^\infty_{\alpha}(\R, M_{p, q}^s(\R^d)) \coloneqq
\set{u \in L^\infty(\R, M_{p, q}^s(\R^d)) \Big| \, \norm{u}_{L^\infty_{\alpha}(\R, M_{p, q}^s(\R^d))} < \infty},
\end{equation*}
where
\begin{equation*}
\norm{u}_{L^\infty_{\alpha}(\R, M_{p, q}^s(\R^d))} \coloneqq \sup_{t \in \R} \, \jb{t}^{\alpha} \norm{u(t, \cdot)}_{M_{p, q}^{s}}.
\end{equation*}

Let us denote by $\pi(u^{m+1})$ any $(m+1)$-time product of $u$ and $\bar{u}$, where $m\in\Z_{+}$. Define also the quantity
\begin{equation}
\label{timeme}
\frac2{\gamma_{m,d}}=
\begin{cases} (d-\frac12)(\frac{m}{2(m+2)}) &,\ \gamma \neq 0,\\
(d-\frac13)(\frac{m}{2(m+2)}) &,\ \gamma=0.\\
\end{cases}
\end{equation}
Futhermore, let $m_0$ denote the positive root of
\begin{equation}
\label{mfirst}
\begin{cases}
(2d-1)x^{2}+(2d-5)x-8 = 0 &, \, \gamma \neq 0,\\
(3d-1)x^{2}+(3d-7)x-12=0 &, \, \gamma = 0.\\
\end{cases}
\end{equation}
The main results are the following theorems.

\begin{theorem}
\label{th1}
Suppose that $d\geq1$, $f(u)=\pm\pi(u^{m+1})$, $m\in\Z_{+}$ with $m>m_{0}$ and $q \in [1, \infty]$.
For $q = 1$, let $s \geq 0$ and for $q > 1$, let $s > \frac{d}{q'}$. Then there exists a $\delta > 0$ such that for any $u_0 \in M_{\frac{m + 2}{m + 1},q}^s(\R^d)$ with $\|u_{0}\|_{M_{\frac{m + 2}{m + 1},q}^{s}}\leq\delta$ the Cauchy problem \eqref{maineq} admits a unique global solution 
\begin{equation}
\label{kka}
u \in L^\infty_{\frac{2}{\gamma_{m, d}}}(\R, M_{2 + m, q}^s (\R^d)).
\end{equation}
\end{theorem}
The restriction on the power of the nonlinearity described in Theorem \ref{th1} is explained in remark \ref{rem:restriction}.

\begin{theorem}
\label{th2}
Suppose that $d\geq2$, $f(u)=\lambda(e^{\rho|u|^2} - 1)u$, $\lambda\in\C$ and $\rho>0$. In addition, let $s\geq0$ if $q=1$ and let $s>\frac{d}{q'}$ if $q\in(1,\infty]$. There exists $\delta>0$ such that for any $u_{0}\in M_{\frac43, q}^{s}(\R^d)$ with $\|u_{0}\|_{M_{\frac43,q}^{s}}\leq\delta$ the Cauchy problem \eqref{maineq} admits a unique global solution $u$ in the space
$L^\infty_{\frac{2}{\gamma_{2,d}}}(\R, M^s_{4, q}(\R^d))$.
\end{theorem}

\begin{remark}
For $q < \infty$, the solution from Theorem \ref{th1} and \ref{th2}
is a continuous function with values in the corresponding modulation
space, i.e. indeed a mild solution. For the more delicate situation
$q = \infty$ see \cite{KPC}.
\end{remark}

The idea of studying the Cauchy problem \eqref{maineq} with such time-decay norm is inspired by \cite{BH}, where the authors considered the NLS and the NLKG equations. As mentioned there, this idea goes back to the work of Strauss, see \cite{STR}. Their results were improved in \cite{TKK1} and \cite{TKK2} where the author considered the nonlinear higher order Schr\"odinger equation
\begin{equation}
\iu \partial_{t}u + \phi(\sqrt{-\Delta})u=f(u),
\end{equation}
where $\phi(\sqrt{-\Delta})=\mathcal F^{-1}\phi(|\xi|)\mathcal F$ and $\phi$ is a polynomial, with initial data $u_{0}$ in a modulation space.

\begin{remark}
Notice that Theorem \ref{th1} does not include the cubic nonlinearity in dimension $d=1$ since $m$ has to be strictly bigger than $m_{0}$ which is the positive root of the quadratics in \eqref{mfirst}, that is $m_{0}=\frac{3+\sqrt{41}}{2}$, if $\gamma\neq0$ and $m_{0}=\frac{4+\sqrt{110}}{4}$, if $\gamma=0$. In both cases $m_{0}>3$. 
\end{remark}

\begin{remark}
In \cite[Theorem 1.1 and Theorem 1.2]{BH}, the authors only considered modulation spaces $M_{p, q}^s(\R^d)$ with $q = 1$. But, by Theorem \ref{th3}, their crucial estimate (6.6) also holds for $q \in (1,\infty]$ and $s > \frac{d}{q'}$. Hence, the statements of their theorems is true in this case too.
\end{remark}

\subsection{Preliminaries}
It is known that for $s > d/q'$ (where $q'$ is the conjugate exponent of $q$) and $p, q \in [1, \infty]$, the embedding
\begin{equation}
\label{yeye}
M_{p,q}^{s}(\R^d)\hookrightarrow C_{\textrm{b}}(\R^d)=\set{f:\R^d\to\C\ \Big| f\ \text{continuous and bounded}},
\end{equation}
is continuous. The same is true for the embedding
\begin{equation}
\label{yeye233}
M_{p_{1}, q_{1}}^{s_{1}}(\R^d)\hookrightarrow M_{p_{2}, q_{2}}^{s_{2}}(\R^d),
\end{equation}
which holds for any $s_1, s_2 \in \R$ and any
$p_1, p_2, q_1, q_2 \in [1, \infty]$ satisfying $p_1 \leq p_2$ and
either
\begin{IEEEeqnarray*}{rCl't'rCl}
q_1 & \leq & q_2 & and & s_1 & \geq & s_2 \\
& & & or & & & \\
q_2 & < & q_1 & and & s_1 & > & s_2 + \frac{d}{q_2} - \frac{d}{q_1}
\end{IEEEeqnarray*}
(see \cite[Proposition 6.8 and Proposition 6.5]{FEI}).

We are going to use the following H\"older type inequality for modulation spaces which appeared in \cite[Theorem 4.3]{CHA} (see also \cite{CHKPmod}).
\begin{theorem}
\label{th3}
Let $d\geq1$ and $1\leq p, p_{1}, p_{2}, q\leq\infty$ such that $\frac1{p}=\frac1{p_{1}}+\frac1{p_{2}}$. For $q=1$ let $s\geq0$ and for $q\in(1,\infty]$ let $s>\frac{d}{q'}$. Then there exists a constant $C=C(d,s,q)>0$ such that 
$$\|fg\|_{M_{p,q}^{s}}\leq C\|f\|_{M_{p_{1},q}^{s}}\|g\|_{M_{p_{2},q}^{s}},$$
for all $f\in M_{p_{1},q}^{s}(\R^d)$ and $g\in M_{p_{2},q}^{s}(\R^d)$. 
\end{theorem}

The propagator of the homogeneous Schr\"odigner equation with higher
order anisotropic dispersion is given by
\begin{equation}
\label{gener}
W(t) = \FT^{(-1)}
e^{\iu(\alpha \abs{\xi}^2 + \beta \xi_1^3 + \gamma \xi_1^4)t} \FT,
\end{equation}
where $\xi = (\xi_{1},\xi')\in\R\times\R^{d-1}$. For the rest of the paper, $A \lesssim B$ shall mean that there is a constant $C > 0$ such that $A \leq C B$. The next dispersive estimate is from \cite[Theorem 1.1]{OB}:

\begin{theorem}
\label{OB1}
Consider $p\in[2,\infty]$ and $f\in L^{p'}(\R^d)$. Then
\begin{equation}
\norm{W(t)f}_{L^p(\R^d)} \lesssim
\abs{t}^{-\mu} \norm{f}_{L^{p'}(\R^d)} 
\end{equation}
where
\begin{equation}
\label{eqn:decay}
\mu = 
\mu(d, \gamma, p) \coloneqq
\begin{cases}
\left(d - \iv{2}\right) \left(\iv{2} - \iv{p}\right) &
\, \gamma \neq 0, \\
\left(d - \iv{3}\right) \left(\iv{2} - \iv{p}\right) &
\, \gamma=0,
\end{cases}
\end{equation}
and the implicit constant is independent of the function $f$ and the
time $t$. 
\end{theorem}
Using this, we claim the following

\begin{theorem}
\label{dispmod}
Consider $s\in\R, p\in[2,\infty]$ and $q\in[1,\infty]$. Then
\begin{equation}
\label{dispmodeq}
\norm{W(t)f}_{M_{p,q}^{s}(\R^{d})} \lesssim
\jb{t}^{-\mu}\norm{f}_{M_{p',q}^{s}(\R^d)}, 
\end{equation}
where $\mu = \mu(d, \gamma, p)$ is as in Equation \eqref{eqn:decay} and the
implicit constant is independent of the function $f$ and the time $t$. 
\end{theorem}
\begin{proof}
The operators $\Box_k$ and $W(t)$ commute and hence we immediately
arrive at
\begin{equation}
\label{eqn:untruncated_decay}
\norm{\Box_k W(t) f}_{L^p(\R^d)} \lesssim 
\abs{t}^{-\mu} \norm{\Box_{k + l} f}_{L^{p'}(\R^d)}
\qquad \forall k \in \Z^d \, \forall t \in \R \setminus \set{0}
\end{equation}
by invoking Theorem \ref{OB1}. Moreover, as $p \in [2, \infty]$, we have
\begin{equation}
\label{eqn:truncated_decay}
\norm{\Box_k W(t) f}_{L^p(\R^d)} \lesssim
\norm{
\sigma_k
e^{\iu(\alpha \abs{\xi}^2 + \beta \xi_1^3 + \gamma \xi_1^4)t}
\hat{f}}_{L^{p'}(\R^d)} \lesssim
\norm{\sigma_k \hat{f}}_{L^p(\R^d)} \lesssim
\norm{\Box_k f}_{L^{p'}(\R^d)}
\end{equation}
for any $k \in \Z^d$ and any $t \in \R$. Above, we used the
Hausdoff-Young inequality for the first and last estimate and the fact
that $\supp(\sigma_k) \subseteq B_{\sqrt{d}}(k)$ for the second
inequality. Taking the minimum of the right-hand sides of
\eqref{eqn:untruncated_decay} and \eqref{eqn:truncated_decay} shows
\begin{equation*}
\norm{\Box_k W(t) f}_{L^p(\R^d)} \lesssim 
\jb{t}^{-\mu} \norm{\Box_k f}_{L^{p'}(\R^d)}
\qquad \forall k \in \Z^d \, \forall t \in \R.
\end{equation*}
Multiplying by the weight $\jb{k}^{s}$ and taking the $l^q(\Z^d)$-norm
on both sides we arrive at the desired estimate.
\end{proof}

\end{section}

\begin{section}{proofs of the main theorems}
In this section we present the proofs of the main theorems.
\begin{proof}[Proof of Theorem \ref{th1}]
For the sake of brevity, let us shorten the notation by setting
\begin{equation}
\label{weirnorm}
\norm{u} \coloneqq
\norm{u}_{L^\infty_{\frac{2}{\gamma_{m, d}}}(\R, M_{m + 2, q}^s(\R^d))}.
\end{equation}
By the Banach fixed-point theorem, it suffices to show that the operator
defined by
\begin{equation}
\label{operweir}
\opT u \coloneqq
W(t) u_0 \pm i\int_0^t W(t - \tau) \left(\pi(u^{m+1})\right) \D{\tau}
\end{equation}
is a contractive self-mapping of the complete metric space
\begin{equation}
\label{metricweir}
M(R) = \set{u \in L^\infty_{\frac{2}{\gamma_{m, d}}}(\R, M_{m + 2, q}^s(\R^d)) \Big| \norm{u} \leq R}
\end{equation}
for some $R \in \R_{+}$. We begin with the self-mapping property and observe that
\begin{equation}
\label{ff}
\norm{\opT u} \leq
\norm{W(t) u_0} + \norm{\int_0^t W(t - \tau) \left(\pi(u^{m + 1})\right) \D{\tau}}.
\end{equation}
Notice, that $\mu(d, \gamma, m + 2) = \frac{2}{\gamma_{m ,d}}$ and hence,
by the dispersive estimate \eqref{dispmodeq}, one obtains
\begin{equation}
\norm{W(t) u_0} =
\sup_{t \in \R} \left[
\jb{t}^{\frac{2}{\gamma_{m,d}}} \norm{W(t) u_0}_{M_{m + 2, q}^s} \right]
\lesssim \norm{u_0}_{M_{\frac{m + 2}{m + 1},q}^s}.
\end{equation}
Introducing the smallness condition
\begin{equation*}
\norm{u_0}_{M_{\frac{m + 2}{m + 1},q}^s(\R^d)} \lesssim \frac{R}{2}
\end{equation*}
leads to $\norm{W(t) u_0} \leq \frac{R}{2}$.

For the integral term we have the upper bound
\begin{equation}
\sup_{t\in\R} \left[ \langle t\rangle^{\frac2{\gamma_{m,d}}}\int_{0}^{t}\langle t-\tau\rangle^{-\frac2{\gamma_{m,d}}}\|\pi(u^{m+1})\|_{M_{\frac{m + 2}{m + 1},q}^{s}} \D{\tau}\right].
\end{equation}
H\"older's inequality for modulation spaces from Theorem \ref{th3} is applicable (due to the assumptions on $s, q$) and yields
\begin{equation}
\|\pi(u^{m+1})\|_{M_{\frac{m + 2}{m + 1},q}^{s}}\lesssim\|u\|_{M_{m + 2,q}^{s}}^{m+1}.
\end{equation}
Furthermore, as $u \in M(R)$, one has
\begin{equation}
\label{eqn:yee455}
\norm{u(\tau, \cdot)}_{M_{m + 2, q}^s} \leq
\jb{\tau}^{-\frac{2}{\gamma_{m ,d}}} \norm{u} \leq
\jb{\tau}^{-\frac{2}{\gamma_{m ,d}}} R \qquad \forall \tau \in \R
\end{equation}
and we obtain the upper bound for the integral term
\begin{equation}
\label{ff1}
R^{m+1} \sup_{t\in\R} \left[\langle t\rangle^{\frac{2}{\gamma_{m,d}}}\int_{0}^{t}\langle t-\tau\rangle^{-\frac2{\gamma_{m,d}}}\langle \tau\rangle^{-\frac{2(m+1)}{\gamma_{m,d}}}\D{\tau} \right].
\end{equation}
To be able to control the individual factors of the integral, we split it into $\int_0^t = \int_0^{\frac{t}{2}} + \int_{\frac{t}{2}}^t$. For the first summand we have
\begin{eqnarray}
\nonumber
\int_{0}^{\frac{t}{2}}\langle t-\tau\rangle^{-\frac2{\gamma_{m,d}}}\langle \tau\rangle^{-\frac{2(m+1)}{\gamma_{m,d}}}\ d\tau&\lesssim&\Big\langle \frac{t}2\Big\rangle^{-\frac2{\gamma_{m,d}}}\frac1{1-\frac2{\gamma_{m,d}}(m+1)}\Big(\Big\langle\frac{t}2\Big\rangle^{1-\frac2{\gamma_{m,d}}(m+1)}-1\Big)\\
\label{ff2}
&\lesssim&
\langle t\rangle^{-\frac2{\gamma_{m,d}}},
\end{eqnarray}
where we used the monotonicity of $\jb{\cdot}$ and the assumption $m>m_{0}$, which implies $\frac{2(m+1)}{\gamma_{m,d}}>1$.
We similarly estimate the second summand by
\begin{equation}
\label{ff3}
\int_{\frac{t}2}^{t}\langle t-\tau\rangle^{-\frac2{\gamma_{m,d}}}\langle \tau\rangle^{-\frac{2(m+1)}{\gamma_{m,d}}}\ d\tau\lesssim\langle t\rangle^{-\frac{2(m+1)}{\gamma_{m,d}}}\int _{\frac{t}2}^{t}\langle t-\tau\rangle^{-\frac2{\gamma_{m,d}}}\lesssim\langle t\rangle^{-\frac2{\gamma_{m,d}}}.
\end{equation}
Putting everything together we arrive at the condition
\begin{equation}
\label{ff4}
\norm{\opT u} \lesssim
\frac{R}{2} + R^{m+1} \overset{!}{\leq} R,
\end{equation}
which is satisfied for suffieciently small $R$.
Similarly we obtain
\begin{equation}
\label{ff5}
\norm{\opT u - \opT v} \lesssim
\left( \norm{u}^m + \norm{v}^m \right) \norm{u - v} \leq 2 R^m \norm{u - v}.
\end{equation}
Hence, under a possibly smaller choice of $R$, the operator $\opT$ is a contraction and the proof is complete.
\end{proof}

\begin{remark}
\label{rem:restriction}
Observe that the restriction $m > m_0$ corresponds to the boundedness of the terms in \eqref{ff2} and \eqref{ff3}.
\end{remark}

\begin{proof}[Proof of Theorem \ref{th2}]
As in the proof of Theorem \ref{th1}, we shorten the notation of the norm by
\begin{equation}
\label{eqee1}
\norm{u} \coloneqq
\norm{u}_{L^\infty_{\frac{2}{\gamma_{2, d}}}(\R, M_{4, q}^s)}
\end{equation}
and introduce the operator
\begin{equation}
\label{eqee2}
\opT u = W(t)u_0 \pm \iu \int_0^t W(t - \tau)(f(u)) \D{\tau},
\end{equation}
which we want to be a contractive self-mapping of the complete metric space $M(R)$ for some $R \in \R_{+}$.
We begin with the self-mapping property. By the definition of the nonlinearity
$f(u)=\lambda(e^{\rho|u|^{2}}-1)u$ we have
\begin{equation}
\label{eqee3}
f(u)=\lambda\sum_{k=1}^{\infty}\frac{\rho^{k}}{k!}|u|^{2k}u.
\end{equation}
Following the proof of Theorem \ref{th1}, we arrive at
\begin{equation}
\label{eqee4}
\norm{\opT u} \lesssim
\norm{u_0}_{M_{\frac43,q}^{s}} +
\sum_{k=1}^{\infty} \sup_{t\in\R} \left[ \langle t\rangle^{\frac2{\gamma_{2,d}}}\int_{0}^{t}\langle t-\tau\rangle^{-\frac{2}{\gamma_{2,d}}}\ \frac{\rho^{k}}{k!}\ \||u|^{2k}u\|_{M_{\frac43,q}^{s}}\D{\tau} \right].
\end{equation}
H\"older's inequality for modulation spaces from Theorem \ref{th3} is applicable (due to the assumptions on $s, q$) and yields the estimate
\begin{equation}
\label{eqee5}
\||u|^{2k}u\|_{M_{\frac43,q}^{s}}\lesssim\|u\|^{3}_{M_{4,q}^{s}}\|u\|_{M_{\infty,q}^{s}}^{2k-2}\lesssim\|u\|_{M_{4,q}^{s}}^{2k+1},
\end{equation}
where in the second inequality we used \eqref{yeye233}, i.e. the embedding
$M_{4,q}^{s}(\R^d)\hookrightarrow M_{\infty,q}^{s}(\R^d)$. Hence, by \eqref{eqn:yee455} for $m = 2$, we obtain
\begin{equation}
\label{eqee6}
\norm{\opT u} \lesssim
\|u_{0}\|_{M_{\frac43,q}^{s}}+\sum_{k=1}^{\infty}\frac{\rho^{k}}{k!}\ R^{2k + 1}\
\sup_{t\in\R} \left[ \langle t\rangle^{\frac2{\gamma_{2,d}}}\int_{0}^{t}\langle t-\tau\rangle^{-\frac2{\gamma_{2,d}}}\langle\tau\rangle^{-(2k+1)\frac2{\gamma_{2,d}}} \D{\tau} \right].
\end{equation}
The supremum above is finite by the same reasoning as in the proof of Theorem \ref{th1} and we therefore arrive at the condition
\begin{equation}
\label{eqee7}
\norm{\opT u} \lesssim
\|u_{0}\|_{M_{\frac43,q}^{s}}+\sum_{k=1}^{\infty}\frac{\rho^{k}}{k!} R^{2k+1} =
\|u_{0}\|_{M_{\frac43,q}^{s}} + \left(R \, e^{\rho R^2} - 1 \right)
\overset{!}{\leq} R.
\end{equation}
Thus, if $\|u_{0}\|_{M_{\frac43,q}^{s}}\lesssim\frac{R}2$ and $R>0$ is sufficiently small, the operator $\mathcal T$ is a self-mapping of the space $M(R)$. The contraction property is proved in a similar way.
\end{proof}

\textbf{Acknowledgments}: The authors gratefully acknowledge financial support by the Deu\-tsche Forschungs\-gemeinschaft (DFG) through CRC 1173.

\end{section}

\end{document}